\documentclass{amsproc}

\newtheorem{theorem}{Theorem}[section]

\newtheorem{corollary}[theorem]{Corollary}

\theoremstyle{definition}
\newtheorem{definition}[theorem]{Definition}
\newtheorem{example}[theorem]{Example}

\theoremstyle{remark}
\newtheorem{remark}[theorem]{Remark}

\numberwithin{equation}{section}

\begin{document}

\title[On some properties of weak solutions]{On some properties of weak solutions \\ to elliptic equations with divergence-free drifts}

\author{Nikolay Filonov}
\address{V.A. Steklov Mathematical Institute, St.-Petersburg, Fontanka 27, 191023, Russia}
\email{filonov@pdmi.ras.ru}
\thanks{Both authors are supported by RFBR grant 17-01-00099-a.}

\author{Timofey Shilkin}
\address{V.A. Steklov Mathematical Institute, St.-Petersburg, Fontanka 27, 191023, Russia}
\email{shilkin@pdmi.ras.ru}
\thanks{The research of the second author
leading to these results has received funding from the People
Programme (Marie Curie Actions) of the European Union's Seventh
Framework Programme FP7/2007-2013/ under REA grant agreement n°
319012 and from the Funds for International Co-operation under
Polish Ministry of Science and Higher Education grant agreement n°
2853/7.PR/2013/2. The  author also thanks the Technische Universit\"
at of Darmstadt for its hospitality.}

\subjclass{35B65}
\date{July  12, 2017.}


\keywords{elliptic equations, weak solutions, regularity}

\begin{abstract}
We discuss the local properties of weak solutions to the equation
$-\Delta u + b\cdot\nabla u=0$. The corresponding theory is
well-known in the case $b\in L_n$, where $n$ is the dimension of the
space. Our main interest is focused on the case $b\in L_2$. In this
case the structure assumption $\operatorname{div} b=0$ turns out to
be crucial.
\end{abstract}

\maketitle

\section{Introduction and Notation}

Assume $n\ge 2$, $\Omega\subset \mathbb  R^n$ is a smooth bounded
domain, $b: \Omega\to \mathbb  R^n$, $f: \Omega\to \mathbb  R$. In
this paper we investigate the   properties of weak solutions $u:
\Omega\to \mathbb  R$ to the following scalar equation
\begin{equation}
-\Delta u + b\cdot \nabla u \ = \ f \qquad \mbox{in} \quad \Omega.
\label{Equation}
\end{equation}
This equation describes the diffusion in a stationary incompressible
flow. If it is not stated otherwise, we always impose the following
conditions
\begin{equation*}
b\in L_2(\Omega) , \qquad f\in W^{-1}_2(\Omega)
\end{equation*}
(see the list of notation at the end of this section). We use the
following

\begin{definition}
\label{Def1} Assume $b\in L_2(\Omega)$, $f\in W^{-1}_2(\Omega)$. The
function $u\in W^1_2(\Omega)$ is called a weak solution to the
equation \eqref{Equation} if the following integral identity holds:
\begin{equation}
\int\limits_\Omega \nabla u\cdot (  \nabla \eta + b  \eta)~dx \ = \
\langle f, \eta\rangle, \qquad \forall~\eta\in C_0^\infty(\Omega).
\label{Integral_Identity_1}
\end{equation}
\end{definition}

Together with the equation \eqref{Equation} one can consider the
formally conjugate (up to the sign of the drift) equation
\begin{equation}
-\Delta u + \operatorname{div}(bu) \ = \ f \qquad \mbox{in} \quad
\Omega. \label{Equation1}
\end{equation}

\begin{definition}
\label{Def2} Assume $b\in L_2(\Omega)$, $f\in W^{-1}_2(\Omega)$. The
function $u\in W^1_2(\Omega)$ is called a weak solution to the
equation \eqref{Equation1} if
\begin{equation}
\int\limits_\Omega (\nabla u   -bu)\cdot   \nabla \eta  ~dx \ = \
\langle f, \eta\rangle, \qquad \forall~\eta\in C_0^\infty(\Omega).
\label{Integral_Identity_2}
\end{equation}
\end{definition}

The advantage of the equation \eqref{Equation1} is that it allows
one to define weak solutions for a drift $b$ belonging to a weaker
class than $L_2(\Omega)$. Namely, Definition \ref{Def2} makes sense
for $u\in W^1_2(\Omega)$ if
\begin{equation}
b\in L_s(\Omega) \qquad \mbox{where} \qquad  s \ = \ \left\{ \
\begin{array}{cl}\frac {2n}{n+2}, \qquad  & n\ge 3, \\
1+\varepsilon,  \ \varepsilon>0, & n=2 .\end{array}\right.
\label{Weak_drift}
\end{equation}
Nevertheless, it is clear that for a divergence-free drift $b\in
L_2(\Omega)$ the Definitions \ref{Def1} and \ref{Def2} coincide.

Together with the equation \eqref{Equation} we discuss boundary
value problems with Dirichlet boundary conditions:
\begin{equation}
\left\{ \quad \gathered -\Delta u +b\cdot \nabla u = f \quad \mbox{in} \quad \Omega , \\
u|_{\partial \Omega} = \varphi .\qquad \endgathered \right.
\label{Dirichlet}
\end{equation}
For weak solutions the boundary condition is understood in the sense
of traces. Assume   $f$ is ``good enough'' and $b \in L_2(\Omega)$,
$\operatorname{div} b = 0$. Our main observation is that the
regularity of solution $u$ inside $\Omega$ can depend on the
behaviour of its boundary values. If the function $\varphi$ is
bounded, then the solution $u$ is also bounded (see Theorem
\ref{Week_Max_Principle_2} below). If the function $\varphi$ is
unbounded on $\partial\Omega$, then the solution $u$ can become
infinite in internal points of $\Omega$ (see Example \ref{Example_1}
below). So, we distinguish between two cases: the case of general
boundary data $\varphi \in W^{1/2}_2(\partial \Omega)$, and the case
of bounded boundary data
\begin{equation}
\varphi \in L_\infty(\partial \Omega)\cap W^{1/2}_2(\partial \Omega)
\label{BC_regular} .
\end{equation}
Discussing the properties  of weak solutions to the problem
\eqref{Dirichlet} we also distinguish between another two cases: in
Section 2 we consider sufficiently regular drifts, namely, $b\in
L_n(\Omega)$, and in Section 3 we focus on the case of drifts $b$
from $L_2(\Omega)$ satisfying $\operatorname{div} b = 0$. Section 4
is devoted to possible ways of relaxation of the condition $b\in
L_n(\Omega)$ in the framework of the regularity theory. In Appendix
for reader's convenience some proofs (most of which are either known
or straightforward) are gathered.

 Together with the  elliptic equation \eqref{Equation} it is possible to consider its
 parabolic analogue
\begin{equation}
\label{113} \partial_t u -\Delta u + b\cdot \nabla u \ = \ f \qquad
\mbox{in} \quad \Omega\times (0,T),
\end{equation}
 but it should be a subject of a separate survey. We address the interested
 readers to the related papers \cite{Zhang}, \cite{Lis_Zhang},
  \cite{NU},  \cite{Sem},  \cite{SSSZ}, \cite{Vicol}, \cite{SVZ} and references there.

In the paper we explore the following notation. For any $a$, $b\in
\mathbb  R^n$ we denote by $a\cdot b$ its scalar product in $\mathbb
R^n$. We denote by $L_p(\Omega)$ and $W^k_p(\Omega)$ the usual
Lebesgue and Sobolev spaces. The space
$\overset{\circ}{W}{^1_p}(\Omega)$ is the closure of
$C_0^\infty(\Omega)$ in $W^1_p(\Omega)$ norm. The negative Sobolev
space $W^{-1}_p(\Omega)$, $p\in (1, +\infty)$,  is the set of all
distributions which are bounded functionals on
$\overset{\circ}{W}{^1_{p'}}(\Omega)$ with $p':=\frac{p}{p-1}$. For
any $f\in W^{-1}_p(\Omega)$ and $w\in
\overset{\circ}{W}{^1_{p'}}(\Omega)$ we denote by $\langle
f,w\rangle$ the value of the distribution $f$ on the function $w$.
We use the notation $W^{1/2}_2(\partial \Omega)$ for the
Slobodetskii--Sobolev space. By $C(\bar \Omega)$ and  $C^\alpha(\bar
\Omega)$, $\alpha\in (0,1)$ we denote the spaces of continuous and
H\" older continuous functions on $\bar \Omega$. The space
$C^{1+\alpha}(\bar \Omega)$ consists of   functions $u$ whose
gradient $\nabla u$ is   H\" older continuous. The index ``loc'' in
notation of the functional spaces $L_{\infty,loc}(\Omega)$,
$C^\alpha_{loc}( \Omega)$, $C^{1+\alpha}_{loc}( \Omega)$ etc implies
that the function belongs to the corresponding functional class over
every compact set which is contained in $\Omega$. The symbols
$\rightharpoonup$ and $\to $ stand for the weak and strong
convergence respectively. We denote by $B_R(x_0)$ the ball in
$\mathbb R^n$ of radius $R$ centered at $x_0$ and write $B_R$ if
$x_0=0$. We write also $B$ instead of $B_1$.

\section{Regular drifts}\label{sreg}
\subsection{Local properties}

For sufficiently regular drifts we have the local H\" older
continuity of a solution.

\begin{theorem}
\label{Holder} Assume
\begin{equation}
\label{210} b \in L_n (\Omega) \quad \text{if} \quad n \ge 3, \qquad
\int\limits_{\Omega} |b|^2\ln(2+|b|^2)~dx \ < \infty \quad \text{if}
\quad n = 2 .
\end{equation}
Let $u\in W^1_2(\Omega)$ be a weak solution to \eqref{Equation} with
$f$ satisfying
\begin{equation*}
f\in L_p(\Omega),  \qquad  p>\frac n2\ . 
\end{equation*}
Then
$$
u\in C^\alpha_{loc}(\Omega) \qquad \mbox{with} \qquad \left\{ \
\begin{array}{cl} \alpha= 2-\frac np, & p<n, \\ \forall~\alpha < 1, & p\ge n.
\end{array}\right.
$$
\end{theorem}

The local H\" older continuity of weak solutions in Theorem
\ref{Holder} with some $\alpha\in (0,1)$ is well-known,  see
\cite[Theorem 7.1]{Stampacchia} or \cite[Corollary 2.3]{NU} in the
case $f\equiv 0$. The H\" older continuity with arbitrary $\alpha\in
(0,1)$ was proved in the case $f\equiv 0$, for example, in
\cite{Filonov}. The extension of this result for non-zero right hand
side is routine.

If $b$ possesses more integrability then the first gradient of a
weak solution is locally H\" older continuous.

\begin{theorem}
Let $b\in L_p(\Omega)$ with $p>n$, and $u\in W^1_2(\Omega)$ be a
weak solution to \eqref{Equation} with $f\in L_p(\Omega)$. Then
$u\in C^{1+\alpha}_{loc}(\Omega)$ with  $\alpha=1-\frac np$.
\end{theorem}

For the proof see \cite[Chapter III, Theorem 15.1]{LU}.

\subsection{Boundary value problem}

We consider the second term $\int\limits_\Omega \nabla u\cdot
b~\eta~dx$ in the equation \eqref{Integral_Identity_1} as a bilinear
form in $\overset{\circ}{W}{^1_2}(\Omega)$. It defines a linear
operator $T: \overset{\circ}{W}{^1_2}(\Omega) \to
\overset{\circ}{W}{^1_2}(\Omega) $ by the relation
\begin{equation}
\label{212} \int\limits_\Omega \nabla (T u) \cdot \nabla \eta~dx =
\int\limits_\Omega \nabla u\cdot b~\eta~dx, \quad \forall~u, \eta\in
\overset{\circ}{W}{^1_2}(\Omega) .
\end{equation}

The following result is well-known.

\begin{theorem}
\label{Operator_is_compact} Let $b$ satisfy \eqref{210}. Then the
operator $T:\overset{\circ}{W}{^1_2}(\Omega)\to
\overset{\circ}{W}{^1_2}(\Omega)$ defined by \eqref{212} is compact.
\end{theorem}

Indeed, if $n\ge 3$ then the estimate
\begin{equation}
\label{2125} \left| \int\limits_\Omega \nabla u\cdot
b~\eta~dx\right| \le C_b \|\nabla u\|_{L_2(\Omega)} \|\nabla
\eta\|_{L_2(\Omega)} \quad \forall~u, \eta\in
\overset{\circ}{W}{^1_2}(\Omega)
\end{equation}
follows by the imbedding theorem and the H\" older inequality. In
the case $n=2$ such estimate can be found for example in \cite[Lemma
4.3]{Filonov}. Next, the operator $T$ can be approximated in the
operator norm by compact linear operators $T_\varepsilon$ generated
by the bilinear forms $\int\limits_\Omega \nabla u\cdot
b_\varepsilon~\eta~dx$ where $b_\varepsilon\in
C^\infty(\bar\Omega)$.

\begin{remark}
The condition $b \in L_2(\Omega)$ in the case $n=2$ is not
sufficient. For example, one can take $\Omega = B_{1/3}$,
$$
b(x) = \frac {x}{|x|^2 \left| \ln |x|\right|^{3/4}}, \qquad u(x) =
\eta(x) = \left| \ln |x|\right|^{3/8} - (\ln 3)^{3/8} .
$$
Then $\int\limits_\Omega \nabla u\cdot b~\eta~dx = \infty$, and
therefore, the corresponding operator $T$ is unbounded.
\end{remark}

\begin{remark}
The issue of boundedness and compactness of the operator $T$ in the
case of the whole space, $\Omega = {\mathbb  R}^n$, is investigated
in full generality in \cite{MV}, see Theorem \ref{t41} below. In
this section we restrict ourselves by considering assumptions on $b$
only in $L_p$--scale.
\end{remark}

Now, the problem \eqref{Dirichlet} with $\varphi\equiv0$ reduces to
the equation $u + T u = h$ in $\overset{\circ}{W}{^1_2}(\Omega)$
with an appropriate right hand side $h$. The solvability of the last
equation follows from the Fredholm theory. Roughly speaking, ``the
existence follows from the uniqueness''.

The uniqueness in the case $b \in L_n(\Omega)$, $n\ge 3$, and
$\operatorname{div} b = 0$ is especially simple. In this situation
$$
\int\limits_\Omega b \cdot \nabla u~ u~dx = 0 \qquad \forall~u\in
\overset{\circ}{W}{^1_2}(\Omega),
$$
and the uniqueness for the problem \eqref{Dirichlet} follows. In the
general case of drifts satisfying \eqref{210} without the condition
$\operatorname{div} b = 0$ the proof of the uniqueness is more
sophisticated. It requires the maximum principle which can be found,
for example, in \cite{NU}, see Corollary 2.2 and remarks at the end
of Section 2 there.

\begin{theorem}
\label{Strong_max_principle} Let $b$ satisfy \eqref{210}. Assume
$u\in W^1_2(\Omega)$ is a weak solution to the problem
\eqref{Dirichlet} with $f\equiv 0$ and $\varphi \in
L_\infty(\partial \Omega)\cap W^{1/2}_2(\partial \Omega)$. Then
either $u\equiv const $ in $\Omega$ or the following estimate holds:
$$
\operatorname{essinf}\limits_{\partial \Omega} \varphi  \ < \ u(x) \
< \ \operatorname{esssup}\limits_{\partial \Omega} \varphi, \qquad
\forall~x\in \Omega.
$$
\end{theorem}

\begin{corollary}
\label{Unique} Let $b$ satisfy \eqref{210}. Then a weak solution to
the problem \eqref{Dirichlet} is unique in the space
$W^1_2(\Omega)$.
\end{corollary}

Now, the solvability of the problem \eqref{Dirichlet} is
straightforward.

\begin{theorem}
\label{Existence_weak_smooth} Let $b$ satisfy \eqref{210}. Then for
any $f \in W^{-1}_2(\Omega)$ and $\varphi \in W^{1/2}_2(\partial
\Omega)$ the problem \eqref{Dirichlet} has the unique weak solution
$u\in W^1_2(\Omega)$, and
$$
\|u\|_{W^1_2(\Omega)} \le C \left(\|f\|_{W^{-1}_2(\Omega)} +
\|\varphi\|_{W^{1/2}_2(\partial\Omega)}\right) .
$$
\end{theorem}

\begin{proof} For $\varphi\equiv0$ Theorem
\ref{Existence_weak_smooth} follows from Fredholm's theory. In the
general case the problem \eqref{Dirichlet} can be reduced to the
corresponding problem with homogeneous boundary conditions for the
function $v:=u-\tilde \varphi$, where $\tilde \varphi$ is some
extension of $\varphi$ from $\partial \Omega$ to $\Omega$ with the
control of the norm $\| \tilde \varphi\|_{W^1_2(\Omega)}\le c\|
\varphi\|_{W^{1/2}_2(\partial \Omega)}$. The function $v$ can be
determined as a weak solution to the problem
\begin{equation}
\left\{ \quad \gathered -\Delta v +b\cdot \nabla v = f +
\Delta\tilde \varphi - b \cdot \nabla  \tilde \varphi
\quad \mbox{in} \quad \Omega, \\
v|_{\partial \Omega} = 0 \qquad \endgathered \right.
\label{Dirichlet_homogeneous}
\end{equation}
Under assumption \eqref{210} the right hand side belongs to
$W^{-1}_2(\Omega)$ due to Theorem \ref{Operator_is_compact}.
\end{proof}

Note that for $n\ge 3$ the problems \eqref{Dirichlet} and
\eqref{Dirichlet_homogeneous} are equivalent only in the case of
``regular'' drifts $b\in L_n(\Omega)$. If $b \in L_2(\Omega)$ and
additionally $\operatorname{div} b = 0$, then $b \cdot \nabla \tilde
\varphi \in W^{-1}_{n'}(\Omega)$, $n'=\frac n{n-1}$, and the
straightforward reduction of the problem \eqref{Dirichlet} to the
problem with homogeneous boundary data is not possible.

Finally, to investigate in Section 3 the problem \eqref{Dirichlet}
with divergence-free drifts from $L_2(\Omega)$ we need the following
maximum estimate.

\begin{theorem}
\label{Weak_Max_Principle} Let $b$ satisfy \eqref{210}. Assume
$\varphi$ satisfies \eqref{BC_regular} and let $u\in W^1_2(\Omega)$
be a weak solution to \eqref{Dirichlet} with some $f \in
L_p(\Omega)$, $p>n/2$. Then

{\rm 1)} $u\in L_\infty(\Omega)$ and
\begin{equation}
\| u\|_{L_\infty(\Omega)} \ \le \  \|\varphi\|_{L_\infty(\partial
\Omega)} + C~\|f\|_{L_p(\Omega)}. \label{Max_Pr_2}
\end{equation}

{\rm 2)} If $\operatorname{div} b = 0$ then $C=C(n,p,\Omega)$ does
not depend on $b$.
\end{theorem}

We believe Theorem \ref{Weak_Max_Principle} is known though it is
difficult for us to identify the precise reference to the statement
we need. So, we present its proof in Appendix.

\begin{remark}
\label{r211} For $n \ge 3$ consider the following example:
$$
\Omega =B, \qquad u(x) = \ln|x|, \qquad b(x) = (n-2)\frac{x}{|x|^2}
.
$$
The statements of Theorem \ref{Holder}, Theorem
\ref{Strong_max_principle} and Corollary \ref{Unique} are violated
for these functions. On the other hand, $-\Delta u + b\cdot \nabla u
= 0$, $u \in \overset{\circ}{W}{^1_2}(\Omega)$ and $b \in
L_p(\Omega)$ for any $p<n$. It means that for non-divergence free
drifts  the condition $b \in L_n(\Omega)$ in \eqref{210} is sharp.
\end{remark}

\begin{remark}
\label{r212} For $n=2$ the condition $b \in L_2(\Omega)$ is not
sufficient. The statements of Theorem \ref{Holder}, Theorem
\ref{Strong_max_principle} and Corollary \ref{Unique} are violated
for the functions
$$
u(x)=\ln\left| \ln |x|\right|, \qquad b(x) = -\frac{x}{|x|^2\ln|x|}
$$
in a ball $\Omega = B_{1/e}$, nevertheless $b \in L_2(\Omega)$.

Converesely, if in the case $n=2$ we assume that $b \in L_2(\Omega)$
and $\operatorname{div} b = 0$, then the estimate \eqref{2125} is
fulfilled (see \cite{MV} or \cite{Filonov}), and all statements of
this section (Theorems \ref{Holder}, \ref{Operator_is_compact},
\ref{Strong_max_principle}, \ref{Existence_weak_smooth} and
\ref{Weak_Max_Principle}) hold true, see \cite{Filonov} or
\cite{NU}. So, this case can be considered as the regular one. See
also Remark \ref{r43} below.
\end{remark}

\section{Non-regular divergence-free drifts}

 In this section we always assume that
$\operatorname{div} b = 0$. It turns out that this assumption plays
the crucial role in local boundedness of weak solutions if one
considers drifts $b\in L_p(\Omega)$ with $p<n$, $n\ge 3$. Recall
that the case $n=2$, $b \in L_2(\Omega)$ and $\operatorname{div} b =
0$ can be considered as a regular case, see Remark \ref{r212}. Thus,
below we restrict ourselves to the case $n\ge 3$.

\subsection{Boundary value problem}

We have the following approximation result.

\begin{theorem}
\label{Approximation} Assume $b\in L_2(\Omega)$, $\operatorname{div}
b=0$, $f\in W^{-1}_2(\Omega)$, and let $u\in W^1_2(\Omega)$ be a
weak solution to \eqref{Equation}. Assume also $b_k \in L_n(
\Omega)$, $\operatorname{div} b_k=0$ is an arbitrary sequence
satisfying
$$
b_k \to b \quad \mbox{in} \quad L_2(\Omega),
$$
and let $u_k\in W^1_2(\Omega)$ be the unique weak solution to the
problem
\begin{equation}
\label{30}
\left\{ \quad \gathered -\Delta u_k +b_k\cdot \nabla u_k = f, \\
u_k|_{\partial \Omega} = \varphi , \endgathered \right.
\end{equation}
where $\varphi = u|_{\partial \Omega}$. Then
\begin{equation}
 u_k\to u \quad \mbox{in} \quad
L_q(\Omega) \quad \mbox{for any} \quad q<\frac n{n-2}.
\label{L_1-conv}
\end{equation}
Moreover, if $\varphi \in L_\infty (\partial\Omega)$ then
\begin{equation}
u_k\rightharpoonup u \quad \mbox{in} \quad W^1_2(\Omega).
\label{Weak_W^1_2}
\end{equation}
 Finally, if
$\varphi\equiv 0$ then  the energy inequality holds:
\begin{equation}
\int\limits_\Omega |\nabla u|^2 ~dx  \ \le  \ \langle f, u\rangle .
\label{Energy_inequality}
\end{equation}
\end{theorem}

The convergence \eqref{L_1-conv} is proved (in its parabolic
version) for $q=1$ in \cite[Proposition 2.4]{Zhang}. Note that the
proof in \cite{Zhang} uses the uniform Gaussian upper bound of the
Green functions of the operators $\partial_t u-\Delta u +b_k\cdot
\nabla u$ (sf. \cite{Aronson}). In Appendix we present an elementary
proof of Theorem \ref{Approximation} based on the maximum estimate
in Theorem \ref{Weak_Max_Principle} and duality arguments.

Theorem \ref{Approximation} has several consequences. The first of
them is the uniqueness of weak solutions, see \cite{Zhang} and
\cite{Zhikov}:

\begin{theorem}
Let $b\in L_2(\Omega)$, $\operatorname{div} b = 0$. Then a weak
solution to the problem \eqref{Dirichlet} is unique in the class
$W{^1_2}(\Omega)$.
\end{theorem}

Indeed, $u$ is a $L_q$-limit of the approximating sequence $u_k$,
and such limit is unique. The alternative proof of the uniqueness
(which is in a sense ``direct'', i.e. it does not hang upon the
approximation result of Theorem \ref{Approximation}) for $b\in
L_2(\Omega)$, $\operatorname{div} b = 0$, can be found in
\cite{Zhikov} (see also some development in
 \cite{Surnachev}). Note that in \cite{Zhikov} it was also shown that
the uniqueness can break for weak solutions to the equation
\eqref{Equation1} if $b$ satisfy \eqref{Weak_drift} (actually a
little better than \eqref{Weak_drift}) and $\operatorname{div} b =
0$, but $b \notin L_2(\Omega)$.

Another consequence of Theorem \ref{Approximation} is the existence
of weak solution.

\begin{theorem}
\label{Existence} Let $b\in L_2(\Omega)$, $\operatorname{div} b =
0$. Then for any $f \in W^{-1}_2 (\Omega)$ and any $\varphi$
satisfying \eqref{BC_regular} there exists a weak solution to the
problem \eqref{Dirichlet}.
\end{theorem}

Theorem \ref{Existence} is proved in Appendix.

Finally, Theorem \ref{Approximation} allows one to establish the
global boundedness of weak solutions whenever the boundary data are
bounded.

\begin{theorem}
\label{Week_Max_Principle_2} Let $b\in L_2(\Omega)$,
$\operatorname{div} b = 0$, $f \in L_p(\Omega)$, $p>n/2$, and
$\varphi $ satisfies \eqref{BC_regular}. Assume $u\in
{W}{^1_2}(\Omega)$ is a weak solution to \eqref{Dirichlet}. Then
$u\in L_\infty(\Omega)$ and
\begin{equation}
\| u\|_{L_\infty(\Omega)} \ \le \  \|\varphi\|_{L_\infty(\partial
\Omega)} + C~ \|f\|_{L_p(\Omega)}, \label{Max_Pr_4}
\end{equation}
where the constant $C=C(n,p,\Omega)$ is independent on $b$.
\end{theorem}

Theorem \ref{Week_Max_Principle_2} is proved in Appendix.

\subsection{Local properties}

Note that any weak solution to \eqref{Equation} belonging to the
class $W^1_2(\Omega)$ can be viewed as a weak solution to the
problem \eqref{Dirichlet} with some $\varphi \in W^{1/2}_2
(\Omega)$.

\begin{theorem}
\label{Local_Boundedness} Assume $\operatorname{div} b = 0$ and
$$
b \in L_p (B) \quad \mbox{where} \quad p=2 \quad \text{if} \quad n =
3 \quad  \mbox{and} \quad p \ > \ \frac n2 \quad \text{if} \quad n
\ge 4 .
$$
Let $u\in W^1_2(B)$ be a weak solution to \eqref{Equation} in $B$
with some $f \in L_q (B)$, $q>n/2$.
 Then $ u\in L_\infty(B_{1/2})$ and
$$
\| u\|_{L_\infty(B_{1/2})} \ \le \ C~\Big( \| u\|_{W^1_2(B)} +
\|f\|_{L_q(B)}\Big)
$$
where the constant $C$ depends only on $n$, $p$, $q$ and
$\|b\|_{L_p(B)}$.
\end{theorem}

Theorem \ref{Local_Boundedness} was proved (in the parabolic
version) in \cite{Zhang}. For the reader's convenience we present
the proof of this theorem in Appendix.

Let us consider the following

\begin{example}\label{Example_1}   Assume $n\ge 4$ and put
$$
u(x) = \ln r, \qquad b = (n-3) \left(~\frac 1r ~{\bf e}_r -
(n-3)~\frac{z}{ r^2} ~{\bf e}_z~\right),
$$
where $r^2 = x_1^2 + ... + x_{n-1}^2$, $z=x_n$, and ${\bf e}_r$,
${\bf e}_z$ are the basis vectors of the corresponding cylindrical
coordinate system in $\mathbb  R^n$. Then $u \in
\overset{\circ}{W}{^1_2}(\Omega)$, and
$$
-\Delta u + b\cdot \nabla u \ = 0.
$$
Next, $\operatorname{div} b = 0$, $b(x) = O(r^{-2})$ near the axis
of symmetry, and hence
$$
b\in L_p(B) \quad \mbox{for any} \quad p  \ < \ \frac {n-1}2 .
$$
\end{example}

Clearly, the assumption $b\in L_2(\Omega)$ leads to the restriction
$n\ge 6$. So, for divergence-free drifts $b\in L_2(\Omega)$ we have
the following picture. Assume $u\in W^1_2(\Omega)$ is a weak
solution to \eqref{Dirichlet} with $f\in L_p (\Omega)$, $p>n/2$.
Theorem \ref{Week_Max_Principle_2} means that
$$
\varphi \in L_{\infty}(\partial\Omega) \cap
W^{1/2}_2(\partial\Omega) \qquad \Longrightarrow \qquad u\in
L_{\infty}(\Omega) \quad \mbox{for any \ $n\ge 2$} .
$$
The Example \ref{Example_1} shows that for general $\varphi$ we have
$$
\varphi \in W^{1/2}_2(\partial\Omega) \quad \Longrightarrow \quad
\left\{ \ \
\begin{array}l  \mbox{if $n\le 3$ then } u\in L_{\infty, loc}(\Omega),
 \\ \mbox{if $n\ge 6$ then it is possible } u\not\in
L_{\infty, loc}(\Omega),    \\ \mbox{if } \mbox{$n= 4,5$ -- open
questions}. \ \
\end{array}\right.
$$

Theorem \ref{Week_Max_Principle_2} and Example \ref{Example_1}
together establish an interesting phenomena: for  drifts $b\in
L_2(\Omega)$, $\operatorname{div} b = 0$, the property of the
elliptic operator in \eqref{Equation} to improve the ``regularity''
of weak solutions (in the sense that every weak solution is locally
bounded) depends on the behavior of a weak solution on the boundary
of the domain. If the values of $\varphi:=u|_{\partial \Omega}$ on
the boundary are bounded then this weak solution must be bounded as
Theorem \ref{Week_Max_Principle_2} says. On the other hand, if the
function $\varphi$ is unbounded on $\partial \Omega$ then the weak
solution can be unbounded even near internal points of the domain
$\Omega$ as Example \ref{Example_1} shows. To our opinion such a
behavior of solutions to an elliptic equation is unexpected.
Allowing  some abuse of language we can say that {\it non-regularity
of the drift can destroy the hypoellipticity of the operator}.

Theorem \ref{Week_Max_Principle_2} impose some restrictions on the
structure of the set of singular points of weak solutions. Namely,
let us define a {\it singular point} of a weak solution as a point
for which the weak solution is unbounded in any its neighborhood,
and then define the {\it singular set} of a weak solution as the set
of all its singular points. It is clear that the singular set is
closed. Theorem \ref{Week_Max_Principle_2} shows that if for some
weak solution its singular set is non-empty then its 1-dimensional
Hausdorff measure  must be positive.

\begin{theorem}
\label{Structure} Let $b\in L_2(\Omega)$, $\operatorname{div} b =
0$, and let $u\in W^1_2(\Omega)$ be a weak solution to
\eqref{Equation} with $f \in L_p(\Omega)$, $p>n/2$. Denote by
$\Sigma\subset \bar \Omega$ the singular set of $u$ and assume
$\Sigma\cap\Omega\not=\emptyset$. Then any point of the set
$\Sigma\cap \Omega$ never can be surrounded by any smooth closed
$(n-1)$-dimensional surface $S\subset \bar \Omega$ such that
$u|_{S}\in L_\infty(S)$. In particular, this means that
\begin{equation}
 \mathcal
H^1(\Sigma)>0, \quad  \Sigma\cap\partial \Omega\not=\emptyset,
\label{Hausdorff}
\end{equation}
where $\mathcal H^1$ is one-dimensional Hausdorff measure in
$\mathbb R^n$.
\end{theorem}

\begin{proof} The first assertion is clear. Let us prove
\eqref{Hausdorff}. Assume $\Sigma\cap \Omega\not=\emptyset$ and
$x_0\in \Sigma\cap \Omega$.  Denote $d:=\operatorname{dist}\{ x_0,
\partial \Omega\}$. Let $z_0\in \partial \Omega$ be a point such
that $|z_0-x_0|=d$ and denote by $[x_0,z_0]$ the straight line
segment connecting $x_0$ with $z_0$. Let us take arbitrary
$\delta>0$ and consider any countable covering of $\Sigma$ by open
balls $\{B_{\rho_i}(y_i)\}$ such that $\rho_i\le \delta$. For any
$i$ denote $r_i:=|x_0-y_i|$. If $r_i\le d$ then denote $z_i:=
[x_0,z_0]\cap
\partial B_{r_i} (x_0)$. By Theorem \ref{Week_Max_Principle_2} for
any $r\le d$ we have $\Sigma\cap \partial B_r(x_0)\not=\emptyset$.
Therefore,
$$
[x_0,z_0] \ \subset \ \bigcup\limits_{r_i\le d} B_{\rho_i}(z_i).
$$
This inclusion means that
$$
\mathcal H^1(\Sigma) \ \ge \ \mathcal H^1\left([x_0,z_0]\right) = d
> 0.
$$
\end{proof}

Theorem  \ref{Structure} in particular implies that no isolated
singularity is possible. This exactly what  Example \ref{Example_1}
demonstrates: the singular set in this case is the axis of symmetry.

Note that the divergence free condition brings significant
improvements into the local boundedness results. Without the
condition $\operatorname{div} b = 0$ one can prove local boundedness
of weak solutions to \eqref{Equation} only for $b\in L_n(\Omega)$
($n\ge 3$), while if $\operatorname{div} b = 0$ the local
boundedness is valid for any $b\in L_p(\Omega)$ with $p>\frac n2$.
Note also that for the moment of writing of this paper we can say
nothing about analogues of neither Theorem \ref{Local_Boundedness}
nor Example \ref{Example_1} if $p\in [\frac{n-1}2, \frac n2]$. {\it
We state this problem as an open question.}

The final issue we need to discuss is the problem of further
regularity of solutions to the equation \eqref{Equation}. The
example of a bounded weak solution which is not locally continuous
was constructed originally in \cite{SSSZ} for $n=3$ and $b \in L_1
(\Omega)$, $\operatorname{div} b = 0$ (actually the method of
\cite{SSSZ} allowed to extend their example for $b \in L_p$, $p\in
[1,2)$). Later the first author in \cite{Filonov} generalized this
example for all $n\ge 3$ and for all $p\in [1,n$).

\begin{theorem}
\label{Filonov_2013} Assume $n\ge 3$, $p < n$. Then there exist
$b\in L_p(B)$ satisfying $\operatorname{div} b = 0$ and a weak
solution $u$ to \eqref{Equation} with $f\equiv 0$ such that $u\in
W^1_2(B)\cap L_\infty(B)$ but $u\not \in C(\bar B_{1/2})$.
\end{theorem}

The latter result shows that if one is interested in the local
continuity of weak solutions then the assumption $b\in L_n(\Omega)$
can not be weakened in the Lebesgue scale and the structure
condition $\operatorname{div} b = 0$ does not help in this
situation.

It is not difficult to construct  also a weak solution to
\eqref{Equation} which is continuous but not H\" older continuous.

\begin{example} \label{Example_2}   Assume $n\ge 4$ and take
$$
u(x) = \frac 1{\ln r}, \quad  b =  \left(\frac {n-3}r - \frac 2{r\ln
r}\right) {\bf e}_r  \ + \  \left( \frac {(n-3)^2}{r^2} -
\frac{2(n-3)}{r^2 \ln r} -\frac{2}{r^2\ln^2 r}~\right)z~{\bf e}_z .
$$
Here $r^2 = x_1^2 + ... + x_{n-1}^2$, $z=x_n$, and ${\bf e}_r$,
${\bf e}_z$ are the basis vectors of the cylindrical coordinate
system. Then $u\in W^1_2(B_{1/2})\cap C(B_{1/2})$, $-\Delta u+b\cdot
\nabla u = 0 $, $\operatorname{div} b=0$ in $B_{1/2}$ and $b\in
L_p(B_{1/2})$ for any $p<\frac {n-1}2$.
\end{example}

\medskip
Thus, for weak solutions of \eqref{Equation} with $b\in
L_2(\Omega)$, $\operatorname{div} b = 0$, in large space dimensions
(at least for $n\ge 6$) the following sequence of implications can
break at any step:
$$
u\in W^1_2(\Omega) \ \ \ \not\!\Longrightarrow \ \ u\in L_{\infty,
loc}(\Omega) \  \ \ \not\!\Longrightarrow \ \ u\in C_{loc}(\Omega) \
\ \ \not\!\Longrightarrow \ \ u\in C^\alpha_{loc}(\Omega).
$$

\section{Beyond the $\boldsymbol{L_p}$--scale}

Theorem \ref{Filonov_2013} shows that in order to obtain the local
continuity of weak solutions to \eqref{Equation} for drifts weaker
than $b\in L_n(\Omega)$ one needs to go beyond the Lebesgue scale.

We start with the question of the boundedness of the operator $T$
defined by the formula \eqref{212}. The necessary and sufficient
condition on $b$ is obtained in \cite{MV} in the case $\Omega =
{\mathbb  R}^n$.

\begin{theorem}
\label{t41} The inequality \eqref{2125} holds true if and only if
the drift $b$ can be represented as a sum $b = b_0 + b_1$, where the
function $b_0$ is such that
\begin{equation}
\label{41} \int\limits_{{\mathbb  R}^n} |b_0|^2 |\eta|^2~dx \ \le \
C~\int\limits_{{\mathbb  R}^n} |\nabla\eta|^2~dx, \qquad
\forall~\eta\in C_0^\infty({\mathbb  R}^n) ,
\end{equation}
$b_1$ is divergence-free, $\operatorname{div} b_1 = 0$, and $b_1 \in
BMO^{-1}(\mathbb  R^n)$. It means that $b_1(x) = \operatorname{div}
A(x)$, $A(x)$ is a skew-symmetric matrix, $A_{ij} = - A_{ji}$, and
$A_{ij} \in BMO ({\mathbb  R}^n)$.
\end{theorem}

Here $BMO (\Omega)$ is the space of functions $f$ with {\it bounded
mean oscillation}, i.e.
$$
\sup_{{\tiny \begin{array}c x\in \Omega \\ 0<r< \infty\end{array}}}
\frac1{r^n} \int\limits_{B_r(x)\cap \Omega} |f(y) - (f)_{B_r(x)\cap
\Omega}|~dy < \infty, \quad \text{where} \quad (f)_{\omega} =
\frac1{|\omega|} \int\limits_{\omega} f(y)~dy.
$$
Clearly, each divergence-free vector $b_1$ can be represented as
$b_1 = \operatorname{div} A$ with a skew-symmetric matrix $A(x)$.

This Theorem mentions that the behaviour of the bilinear form
$\int\limits_\Omega \nabla u\cdot b~\eta~dx$ already distinguish
between general drifts and divergence-free drifts. First, let us
discuss general drifts. If $b$ satisfies \eqref{210} then it
satisfies the estimate \eqref{41} too. But we can not use the
condition \eqref{41} instead \eqref{210} for the regularity theory,
as the example of Remark \ref{r211} shows. Indeed, for functions
satisfying
\begin{equation}
\label{42} |b(x)| \ \le \ \frac{C}{|x|}
\end{equation}
the estimate \eqref{41} is fulfilled by the Hardy inequality.

On the other hand, the case of the drift $b$ having a one-point
singularity (say, at the origin) with the asymptotics which includes
homogeneous of degree $-1$ functions like \eqref{42}, is also
interesting. There are several papers, see \cite{Lis_Zhang},
\cite{Sem}, \cite{SSSZ} and \cite{NU}, dealing with different
classes of divergence-free drifts which cover \eqref{42}. All these
papers contain also the results for parabolic equation \eqref{113},
but we discuss only (simplified) elliptic versions of them. We
address the interested readers to the original papers.

The approach of \cite{SSSZ} seems to be the most general one. Assume
$b \in BMO^{-1}(\Omega)$ and $\operatorname{div} b = 0$. In this
case we understand the equation $-\Delta u + b\cdot \nabla u \ = 0$
in the sense of the integral identity
\begin{equation}
\label{43} \int_\Omega \left(\nabla u \cdot \nabla \eta + A \nabla u
\cdot \nabla \eta\right) dx = 0 \qquad \forall~\eta\in
C_0^\infty(\Omega),
\end{equation}
where the skew-symmetric matrix $A \in BMO (\Omega)$ is defined via
$\operatorname{div} A(x) = b(x)$.

\begin{theorem}
\label{t42} Let $b \in BMO^{-1}(\Omega)$ and $\operatorname{div} b =
0$. Then

{\rm 1)} The maximum principle holds. If $u\in W^1_2(\Omega)$
satisfies \eqref{43} and  $\varphi :=
\left.u\right|_{\partial\Omega}$ is bounded, then
$\|u\|_{L_\infty(\Omega)} \le
\|\varphi\|_{L_\infty(\partial\Omega)}$. In particular, the weak
solution to \eqref{Dirichlet} is unique.

{\rm 2)} Any weak solution $u$ to \eqref{Equation} is H\" older
continuous, $u \in C^\alpha_{loc} (\Omega)$ for some $\alpha > 0$.
\end{theorem}

For the proof see \cite{NU} or \cite{SSSZ}. The regularity theory
developped in Section \ref{sreg} is slightly better as it guarantees
that weak solutions are locally H\" older continuous with {\it any}
exponent $\alpha < 1$. Nevertheless, Theorem \ref{t42} means that
divergence-free drifts from $BMO^{-1}$ can  be also considered as
regular ones.

\begin{remark}
\label{r43} Note that the case $n=2$, $b \in L_2(\Omega)$,
$\operatorname{div} b = 0$, is the particular case of this
situation. Indeed, such drifts can be represented as a
vector-function with components $b_1 = \partial_2 h$, $b_2 = -
\partial_1 h$, where $h$ is a scalar function $h \in W^1_2(\Omega)$.
By the imbedding theorem $W^1_2(\Omega) \subset BMO(\Omega)$ we have
$$
A(x) = \left( \begin{array}{cc}
0 & -h(x) \\
h(x) & 0 \end{array} \right) \in BMO(\Omega) .
$$
\end{remark}

\section{Appendix}
 First we prove  Theorem
\ref{Weak_Max_Principle}.

\begin{proof} We present the proof in the case $n\ge 3$ only. The case $n=2$
differs from it by routine technical details.

1) The statement similar to our estimate \eqref{Max_Pr_2} (for more
general equations) can be found in \cite{Stampacchia}. In
particular, in \cite[Theorem 4.2]{Stampacchia} the following
estimate for weak solutions to the problem
\begin{equation}
\left\{ \quad \gathered -\Delta u +b\cdot \nabla u = f \quad \mbox{in} \quad \Omega, \\
u|_{\partial \Omega} = 0, \qquad \endgathered \right.
\label{Dirichlet_3}
\end{equation}
  was proved:
 \begin{equation}
\| u\|_{L_\infty(\Omega)} \ \le \ C~\Big( ~\| f\|_{L_p(\Omega)} +
\|u\|_{L_2(\Omega)}~\Big). \label{Stamp}
 \end{equation}
On the other hand,
\begin{equation}
\|u\|_{W^1_2(\Omega)} \ \le \ C~\|f\|_{W^{-1}_2(\Omega)}
\label{Zero_RHS}
\end{equation}
due to Theorem \ref{Existence_weak_smooth}. Hence we can exclude the
weak norm of $u$ from the right hand side of \eqref{Stamp} and
obtain the estimate \eqref{Max_Pr_2} in the case $\varphi\equiv 0$.
In general case we can split a weak solution $u$ of the problem
\eqref{Dirichlet} as $u=u_1+u_2$, where $u_1$ is a weak solution of
\eqref{Dirichlet_3} and $u_2$ is a weak solution to the problem
\eqref{Dirichlet} with the boundary data $\varphi$ and zero right
hand side. For $u_1$ we have \eqref{Zero_RHS} and for $u_2$ we have
$\|u_2\|_{L_\infty(\Omega)} \ \le \ \|\varphi\|_{L_\infty(\partial
\Omega)}$ by Theorem \ref{Strong_max_principle}.

2) As $b \in L_n(\Omega)$ we can complete the integral identity
\eqref{Integral_Identity_1} up to the test functions $\eta\in
\overset{\circ}{W}{^1_2}(\Omega)$. Denote
$k_0:=\|\varphi\|_{L_\infty(\partial\Omega)}$ and assume $k\ge k_0$.
Take in \eqref{Integral_Identity_1} $\eta=(u-k)_+$, where we denote
$(u)_+:=\max\{ u,0\}$. As $k\ge k_0$ we have $\eta\in
\overset{\circ}{W}{^1_2}(\Omega)$ and $\nabla \eta =
\chi_{A_k}\nabla u$ where $\chi_{A_k}$ is the characteristic
function of the set
$$A_k \ := \ \{ ~x\in \Omega: ~u(x)>k~\}.$$
We obtain the identity
$$
\int\limits_{A_k} |\nabla u|^2~dx \ + \ \int\limits_{A_k} b\cdot
(u-k)\nabla u ~dx \ = \ \int\limits_{A_k} f (u-k)~dx .
$$
The second term vanishes
$$
\int\limits_{A_k} b\cdot (u-k)\nabla u ~dx \ = \ \frac
12~\int\limits_\Omega b\cdot \nabla |(u-k)_+|^2~dx \ = \ 0,
$$
as $\operatorname{div} b = 0$, and hence
$$
\int\limits_{A_k} |\nabla u|^2~dx  \ = \ \int\limits_{A_k} f
(u-k)~dx, \qquad \forall~k\ge k_0.
$$
The rest of the proof goes as in the usual elliptic theory. Applying
the imbedding theorem we obtain
$$
\left(~\int\limits_{A_k} |\nabla u|^2~dx\right)^{\frac 12}  \ \le \
C(n)~\left(~ \int\limits_{A_k} |f|^{\frac
{2n}{n+2}}~dx\right)^{\frac {n+2}{2n}},
$$
and using the H\" older inequality we get
$$
\|f\|_{L_{\frac {2n}{n+2}}(A_k)} \ \le \ |A_k|^{\frac {n+2}{2n} -
\frac 1p}~ \|f\|_{L_p(A_k)} .
$$
So we arrive at
$$
\int\limits_{A_k} |\nabla u|^2~dx \ \le \ C(n)~
\|f\|_{L_p(\Omega)}^2~|A_k|^{1-\frac 2n+\varepsilon} , \qquad
\forall~k\ge k_0,
$$
where  $\varepsilon:=2\left(\frac 2n-\frac 1p\right)>0$. This
inequality yields the following estimate, see \cite[Chapter II,
Lemma 5.3]{LU},
$$
\operatorname{esssup}\limits_\Omega (u-k_0)_+ \ \le \ C(n,p,
\Omega)~\|f\|_{L_p(\Omega)} .
$$
The estimate of $\operatorname{essinf}\limits_\Omega u$ can be
obtained in a similar way if we replace $u$ by $-u$. \end{proof}

In order to prove Theorem \ref{Approximation} we need some auxiliary
results.

\begin{theorem}
\label{L_1-theorem} Assume $n\ge 3$, $b\in C^\infty(\bar \Omega)$,
$\operatorname{div} b=0$ in $\Omega$, $f\in L_1(\Omega)$, and assume
$u\in \overset{\circ}{W}{^1_2}(\Omega)$
 is a weak solution of \eqref{Dirichlet} with $\varphi\equiv 0$.
Then for any $q\in \big[1,\frac n{n-2}\big)$
 the following estimate holds:
\begin{equation}
\| u\|_{L_q(\Omega)} \ \le \ C(n,q,\Omega)~\|f\|_{L_1(\Omega)}.
\label{L_1-estimate}
\end{equation}
\end{theorem}

\begin{proof} Assume $q\in \big(1,\frac n{n-2}\big)$. By
duality we have
$$
\| u\|_{L_q(\Omega)} \ = \ \sup\limits_{g\in L_{q'}(\Omega), \
\|g\|_{L_{q'}(\Omega)}\le 1} \ \int\limits_\Omega ug~dx ,
$$
where $q':=\frac q{q-1}$, $q'>\frac n2$. For any  $g\in
L_{q'}(\Omega)$ denote by $w_g \in W^2_{q'}(\Omega)$ a solution to
the problem
$$
\left\{ \ \  \gathered -\Delta w_g - b\cdot \nabla w_g  \ = \
g\qquad \mbox{in} \quad \Omega, \\ w_g|_{\partial \Omega} \ = \ 0.
\qquad \qquad
\endgathered \right.
$$
From Theorem \ref{Weak_Max_Principle} we conclude that for  $w_g$
the following estimate holds:
$$
\| w_g\|_{L_\infty(\Omega)} \ \le \ C(n, q,
\Omega)~\|g\|_{L_{q'}(\Omega)}.
$$
Integrating by parts we obtain
$$
\int\limits_\Omega ug~dx \ = \ \int\limits_\Omega u(-\Delta w_g -
b\cdot \nabla w_g)~dx \ = \ \int\limits_\Omega \nabla u \cdot
(\nabla w_g + bw_g)~dx \ = \ \int\limits_\Omega fw_g~dx .
$$
Then for any  $g\in L_{q'}(\Omega)$ such that
$\|g\|_{L_{q'}(\Omega)}\le 1$ we get
$$
\int\limits_\Omega ug~dx \ \ = \
 \int\limits_\Omega fw_g~dx \ \le \
\|f\|_{L_1(\Omega)}\| w_g\|_{L_\infty(\Omega)} \ \le \ C(n, q,
\Omega)~\|f\|_{L_1(\Omega)} .
$$
Hence we obtain \eqref{L_1-estimate}. \end{proof}

Another auxiliary result we need is the following extension theorem.

\begin{theorem}
\label{Extension} Assume $\Omega\subset \mathbb  R^n$ is a bounded
domain of class $C^1$. Then there exists a bounded linear extension
operator $T: L_\infty(\partial \Omega)\cap W^{1/2}_2(\partial
\Omega) \to L_\infty(\Omega)\cap W^{1}_2( \Omega)$ such that
$$
T\varphi|_{\partial \Omega} \ = \ \varphi, \qquad \forall~\varphi
\in L_\infty(\partial \Omega)\cap W^{1/2}_2(\partial \Omega),
$$
$$
\| T\varphi\|_{W^1_2(\Omega)} \ \le \
C(\Omega)~\|\varphi\|_{W^{1/2}_2(\partial \Omega)}, \qquad
\|T\varphi\|_{L_\infty(\Omega)} \ \le \
C(\Omega)~\|\varphi\|_{L_\infty(\partial\Omega)} .
$$
\end{theorem}

\begin{proof} For the sake of completeness we briefly
recall the proof of Theorem \ref{Extension}. After the localization
and flattening of the boundary it is sufficient to construct the
extension operator from $\mathbb  R^{n-1}$ to $\mathbb
R^n_+:=\mathbb R^{n-1}\times (0, +\infty)$. Then we can take the
standard operator
$$
(T\varphi)(x', x_n) \ = \ \eta(x_n)~\int\limits_{\mathbb  R^{n-1}}
\varphi(x'-x_n\xi')\psi(\xi')~d\xi', \qquad (x',x_n)\in \mathbb
R^n_+,
$$
where $x':=(x_1, \ldots, x_{n-1}) \in \mathbb  R^{n-1}$, $\eta\in
C_0^\infty(\mathbb  R)$, $\eta(0)=1$, $\psi\in C_0^\infty(\mathbb
R^{n-1})$, $\int\limits_{\mathbb  R^{n-1}}\psi(\xi')~d\xi' = 1$.
This operator is bounded from $W^{1/2}_2(\mathbb  R^{n-1})$ to
$W^1_2(\mathbb  R^n_+)$ and also from $L_\infty(\mathbb  R^{n-1})$
to $L_\infty(\mathbb  R^n_+)$. More details can be found in
\cite{BIN}.
\end{proof}

Now we can give an elementary proof of Theorem \ref{Approximation}.

\begin{proof} The function
$v_k:=u_k-u\in \overset{\circ}{W}{^1_2}(\Omega)$ is a weak solution
to the problem
$$
\left\{ \quad \gathered -\Delta v_k +b_k\cdot \nabla v_k \ = \ f_k
\quad \mbox{in}\quad \Omega ,\\ v_k|_{\partial \Omega} \ = \ 0
,\qquad \qquad
\endgathered \right.
$$
where
$$
f_k:=(b-b_k)\cdot \nabla u, \qquad f_k\in L_1(\Omega), \qquad
\|f_k\|_{L_1(\Omega)} \ \to \ 0.
$$
Assume $q\in \big[1, \frac n{n-2}\big)$. By Theorem
\ref{L_1-theorem} we have
$$
\| v_k\|_{L_q(\Omega)} \ \le \ C(n, \Omega)~\|f_k\|_{L_1(\Omega)} \
\to \ 0 ,
$$
and hence \eqref{L_1-conv} follows.

Now assume additionally $\varphi\in L_\infty(\partial\Omega)$.
Denote $\tilde \varphi:=T\varphi$ where $T$ is the extension
operator from Theorem \ref{Extension}. Taking in the integral
identity \eqref{Integral_Identity_2} for $u_k$ and $b_k$ the test
function $\eta = u_k-\tilde \varphi \in
\overset{\circ}{W}{^1_2}(\Omega)$ we obtain
$$
\int\limits_\Omega |\nabla u_k|^2~dx \ - \ \int\limits_\Omega u_k
b_k\cdot \nabla (u_k-\tilde \varphi)~dx \ = \ \int\limits_\Omega
\nabla u_k \cdot \nabla \tilde \varphi~dx \ + \  \langle f,
u_k-\tilde \varphi\rangle .
$$
Using the condition $\operatorname{div} b_k=0$ we get
$$
\int\limits_\Omega u_k b_k\cdot \nabla (u_k-\tilde \varphi)~dx \ = \
\int\limits_\Omega \tilde \varphi b_k\cdot \nabla (u_k-\tilde
\varphi)~dx .
$$
Therefore,
$$
\gathered \|\nabla u_k\|_{L_2(\Omega)}^2 \ \le \ \Big( \|\tilde
\varphi\|_{L_\infty( \Omega)}\|b_k \|_{L_2(\Omega)} \ + \ \|
f\|_{W^{-1}_2(\Omega)}\Big) \Big(\|u_k\|_{W^1_2(\Omega)}+ \|\tilde
\varphi\|_{W^1_2(\Omega)}\Big) \ + \\ + \ \|\nabla
u_k\|_{L_2(\Omega)}\|\nabla \tilde \varphi\|_{L_2(\Omega)} .
\endgathered
$$
Applying  Friedrichs' and Young's inequalities we obtain the
estimate
\begin{equation}
\| u_k\|_{W^1_2(\Omega)} \ \le \ C, \label{W^1_2-estimate}
\end{equation}
with a constant $C$ independent on $k$. As the convergence
\eqref{L_1-conv} is already established, from \eqref{W^1_2-estimate}
we derive \eqref{Weak_W^1_2}.

Finally, if $\varphi\equiv 0 $ then we have the energy identities
for $u_k$
$$
\int\limits_\Omega |\nabla u_k|^2 \ = \ \langle f, u_k\rangle ,
$$
and using the weak convergence \eqref{Weak_W^1_2} we arrive at
\eqref{Energy_inequality}. \end{proof}

Now we turn to the proof of Theorem \ref{Existence}.

\begin{proof}
We take a sequence $b_k\in C^\infty(\bar \Omega)$,
$\operatorname{div} b_k=0$, such that $b_k\to b$ in $L_2(\Omega)$.
Let $u_k\in W^1_2(\Omega)$ be a weak solution to the problem
\eqref{30}. Repeating the arguments in the proof of Theorem
\ref{Approximation}, we obtain the estimate \eqref{W^1_2-estimate}
with a constant $C$ independent on $k$. Using this estimate we can
extract a subsequence satisfying \eqref{Weak_W^1_2} for some $u\in
W^1_2(\Omega)$. The weak convergence \eqref{Weak_W^1_2} and the
strong convergence $b_k\to b$ in $L_2(\Omega)$ allow us to pass to
the limit in the integral identities \eqref{Integral_Identity_1}
corresponding to $u_k$ and $b_k$. Therefore, $u$ is a weak solution
to \eqref{Dirichlet}. \end{proof}

Now  we present the  proof of Theorem \ref{Week_Max_Principle_2}.

\begin{proof}
Let $b_k$ be smooth divergence-free vector fields such that $b_k\to
b$ in $L_2(\Omega)$. Denote by $u_k$ the weak solution to the
problem \eqref{30}. By Theorem \ref{Weak_Max_Principle}
\begin{equation}
\| u_k\|_{L_\infty(\Omega)} \ \le \  \|\varphi\|_{L_\infty(\partial
\Omega)} + C~\|f\|_{L_p(\Omega)} \label{Max_Pr_3}
\end{equation}
with the constant $C$ depending only on $n$, $p$ and $\Omega$. From
Theorem \ref{Approximation} we have the convergence $u_k\to u$ in
$L_1(\Omega)$ and hence we can extract a subsequence (for which we
keep the same notation) such that
$$
u_k \to u \quad \mbox{a.e. in} \quad \Omega.
$$
Passing to the limit in \eqref{Max_Pr_3} we obtain \eqref{Max_Pr_4}.
\end{proof}

Finally we give the proof of Theorem \ref{Local_Boundedness}.

\begin{proof}
To simplify the presentation we give the proof only in the case
$f\equiv 0$. The extension of the result for non-zero right hand
side can be done by standard methods, see \cite[Theorem
4.1]{Han_Lin} or \cite{Preprint_Darmstadt}.  First we derive the
estimate
\begin{equation}
\|  u\|_{L_{\infty}(B_{1/2})} \ \le \ C
\left(1+\|b\|_{L_p(B)}\right)^\mu~ \| u\|_{L_{2p'}(B)},  \qquad
p':=\frac p{p-1} \label{We_want}
\end{equation}
(with some positive constants $C$ and $\mu$ depending only on $n$
and $p$)   under additional assumption  $u\in C^\infty(B)$. We
explore Moser's iteration technique, see \cite{Moser}. Assume $\beta
\ge 0$ is arbitrary  and let $\zeta\in C_0^\infty(B)$ be a cut-off
function. Take a test function  $ \eta = \zeta^2|u|^{\beta} u $  in
the identity \eqref{Integral_Identity_1}. Denote $ w \ := \
|u|^{\frac {\beta+2} 2}$. Then after integration by parts and some
routine calculations we obtain the inequality
\begin{equation}
\gathered  \int\limits_B |\nabla(\zeta w)|^2~dx \ \le \ C ~
\int\limits_B  |w|^2~\Big(|\nabla \zeta|^2  +
 |b|~ |\nabla \zeta|\Big)
~dx
\endgathered
\label{Instead}
\end{equation}
 Applying the imbedding theorem and the H\" older
inequality and choosing the test function $\zeta$ in an appropriate
way  we arrive at the inequality
$$
\| w\|_{L_{\frac {2n}{n-2}}(B_r)} \ \le \ C\left( \frac{1 }{R-r} \ +
\ \|b\|_{L_p(B_R)} \right)~ \| w\|_{L_{2p'}(B_R)},
$$
which holds for any $\frac 12\le r<R\le 1$. Note that $\frac
{2n}{n-2}>2p'$ as $p>\frac n2$ if $n\ge 4$ and $p=2$ if $n=3$. The
latter inequality gives us the estimate
\begin{equation}
\|  u\|_{L_{\frac{n\gamma}{n-2} }(B_r)} \ \le \ C^{\frac{2}\gamma}
\left(\frac{1}{R-r}  +  \|b\|_{L_p(B_R)} \right)^{\frac 2{ \gamma}}~
\| u\|_{L_{p'\gamma}(B_R)} \label{Iterate}
\end{equation}
with an arbitrary $\gamma\ge 2$, $\gamma:=\beta+2$. Denote
$s_0=2p'$, $s_m:=\chi s_{m-1} $, where
$\chi:=\frac{n(p-1)}{p(n-2)}$, and denote also $R_m= \frac 12
+\frac1{2^{m+1}}$. Taking in \eqref{Iterate} $r=R_m$, $R=R_{m-1}$,
$\gamma=\frac{s_{m-1}}{p'}$ we obtain
$$
\gathered  \| u\|_{L_{s_m }(B_{R_m})}  \ \le \   \Big(    C~2^{m+1}
\ + \ C\|b\|_{L_p(B)}
 \Big)^{\frac 1{ \chi^{m-1}}}~ \|
u\|_{L_{s_{m-1}}(B_{R_{m-1}})}
\endgathered
$$
Iterating this inequality we arrive at \eqref{We_want}.

Now we need to get rid of the assumption   $u\in C^\infty(B)$.
Assume $u\in W^1_2(B)$ is an arbitrary weak solution to
\eqref{Equation}. Let $\zeta\in C_0^\infty(B)$ be a cut-off function
such that $\zeta\equiv 1$ on $B_{5/6}$ and denote $v:=\zeta u$. Then
$v$ is a weak solution to the boundary value problem
$$
\left\{ \quad \gathered -\Delta v +b\cdot \nabla v = g \quad \mbox{in} \quad B \\
v|_{\partial B} = 0 \qquad \endgathered \right.
$$
where
$$
g:=   -u\Delta \zeta - 2\nabla u \cdot \nabla \zeta+ bu\cdot \nabla
\zeta.
$$
Note that $g \equiv 0$     and $v\equiv u$ on $B_{5/6}$. As $b\in
L_p(B)$ with  $p>\frac n2$ we have  $g\in W^{-1}_2(B)$. Now we take
a sequence $b_k\in C^\infty(\bar B)$, $\operatorname{div} b_k=0$,
such that $b_k\to b$ in $L_p(B)$  and let $v_k$ be the  weak
solution to the problem
$$
\left\{ \quad \gathered -\Delta v_k +b_k\cdot \nabla v_k = g \quad \mbox{in} \quad B \\
v_k|_{\partial B} = 0 \qquad \endgathered \right.
$$
From Theorem \ref{Approximation} we have $v_k\rightharpoonup v$ in
$W^1_2(B)$ and as $p>\frac n2$ we can extract a subsequence (for
which we keep the same notation) such that $v_k\to v$ a.e. in $B$
and $v_k\to v$ in $L_{2p'}(B)$. As $g\equiv 0$ on $B_{5/6}$ from the
usual elliptic theory (see \cite{LU}) we conclude that $v_k\in
C^\infty(B_{5/6})$. Applying \eqref{We_want} (with the obvious
modification in radius) we obtain the estimate
$$
\|  v_k \|_{L_{\infty}(B_{1/2})} \ \le \ C
 \left(1+\|b_k\|_{L_p(B)}\right)^\mu~ \| v_k\|_{L_{2p'}(B_{3/4})}.
$$
Hence $v_k$ are equibounded on $B_{1/2}$. Passing to the limit in
the above inequality and taking into account that $v=u$ on $B_{
5/6}$ we obtain
$$
\|  u \|_{L_{\infty}(B_{1/2})} \ \le \ C
 \left(1+\|b\|_{L_p(B)}\right)^\mu~\| u\|_{L_{2p'}(B_{ 3/4})}.
$$
To conclude the proof we remark that for $p>\frac n2$ from the
imbedding theorem we have
$$
\| u\|_{L_{2p'}(B)} \ \le \ C(n,p)~\|u\|_{W^1_2(B)}.
$$
\end{proof}

\bibliographystyle{amsalpha}

\begin{thebibliography}{A}


\bibitem[A]{Aronson}
D. G. Aronson, \textit{Non-negative solutions of linear parabolic
equations}, Ann. Scuola Norm. Sup. Pisa, \textbf{22} (1968), pp.
607-694.

\bibitem[BIN]{BIN}  O. V. Besov, V. P. Il'in, S. M. Nikol'skii, \textit{Integral
representations of functions and imbedding theorems}, Moscow, 1975.

\bibitem[F]{Filonov}
N. Filonov,  \textit{On the regularity of solutions to the equation}
$-\Delta u+b\cdot \nabla u=0$, Zap. Nauchn. Sem. of Steklov Inst.
\textbf{410} (2013),  168-186; reprinted in J. Math. Sci. (N.Y.)
\textbf{195} (2013), no. 1, 98-108.

\bibitem[FSh]{Preprint_Darmstadt}
N. Filonov, T. Shilkin, \textit{On the local boundedness of weak
solutions to elliptic equations with divergence-free drifts},
Preprint  2714, Technische Universit\" at  Darmstadt, 2017.

\bibitem[HL]{Han_Lin}
Q. Han, F. H. Lin, \textit{Elliptic partial differential equations,
Courant Lecture Notes in Mathematics}, AMS, 1997.

\bibitem[LU]{LU}
O. A. Ladyzhenskaya, N. N. Uraltseva, \textit{Linear and quasilinear
equations of elliptic type}, Academic Press, 1968.

\bibitem[LZ]{Lis_Zhang}
V. Liskevich, Q. S. Zhang, \textit{Extra regularity for parabolic
equations with drift terms}, Manuscripta Math. \textbf{113} (2004),
no. 2, 191-209.

\bibitem[M]{Mazya}
V. G. Mazja, \textit{Sobolev Spaces}, Springer, 1985.

\bibitem[MV]{MV}
V. G. Mazja, I. E. Verbitskiy, \textit{Form boundedness of the
general second-order differential operator}, Comm. Pure Appl. Math.
\textbf{59} (2006), 1286-1329.

\bibitem[Mo]{Moser}
J. Moser, \textit{A new proof of De Giorgi's theorem concerning the
regularity problem for elliptic differential equations}, Comm. Pure
and Appl. Math.,  \textbf{13} (1960), no. 3, pp. 457-468.

\bibitem[NU]{NU}
A. I. Nazarov, N. N. Uraltseva, \textit{The Harnack inequality and
related pro\-per\-ties of solutions of elliptic and parabolic
equations with divergence-free lower-order coefficients},  St.
Petersburg Math. J. \textbf{23} (2012), no. 1, 93-115.

\bibitem[Sem]{Sem}
 Y. A. Semenov, \textit{Regularity theorems for parabolic
equations}, J. Funct. Anal.  \textbf{231} (2006), no. 2, 375-417.

\bibitem[SSSZ]{SSSZ}
 G. Seregin, L. Silvestre, V. Sverak, A. Zlatos,  \textit{On
divergence-free drifts}, J. Differential Equations  \textbf{252}
(2012), no. 1, 505-540.

\bibitem[SV]{Vicol}
L. Silvestre, V. Vicol,  \textit{H\" older continuity for a
drift-diffusion equation with pressure}, Ann. Inst. H. Poincare
Anal. Non Lineaire   \textbf{29}  (2012), no. 4, 637-652.

\bibitem[SVZ]{SVZ}
  L. Silvestre, V. Vicol, A. Zlatos, \textit{On the loss of
continuity for super-critical drift-diffusion equations}, Arch.
Ration. Mech. Anal. \textbf{207} (2013), no. 3, 845-877.

\bibitem[St]{Stampacchia}
  G. Stampacchia, \textit{Le probl\` eme de Dirichlet pour les
\' equations elliptiques du second ordre \` a coefficients
discontinus.} (French) Ann. Inst. Fourier (Grenoble)   \textbf{15}
(1965) fasc. 1, 189-258.

\bibitem[Su]{Surnachev}
  M. D.~Surnachev, \textit{On the uniqueness of a solution to a
stationary convection-diffusion equation with a generalized
divergence-free drift},  arXiv:1706.00389, 2017.

\bibitem[Z]{Zhang}
  Q. S. Zhang, \textit{A strong regularity result for parabolic
equations}, Commun. Math. Phys.,   \textbf{244}  (2004), no. 2, pp.
245-260.

\bibitem[Zhi]{Zhikov}
  V. V. Zhikov, \textit{Remarks on the uniqueness of the
solution of the Dirichlet problem for a second-order elliptic
equation with lower order terms}, Funct. Anal. Appl.  \textbf{38}
(2004), no. 3, 173-183.


\end{thebibliography}

\end{document}